\RequirePackage{fix-cm}
\documentclass[smallextended]{svjour3}       
\smartqed  
\usepackage{graphicx}
\usepackage{amsmath}
\usepackage{amsfonts}
\usepackage{mathtools}
\DeclarePairedDelimiter\ceil{\lceil}{\rceil}
\DeclarePairedDelimiter\floor{\lfloor}{\rfloor}

\begin{document}

\title{Approximation by Exponential Type Neural Network Operators
}

\titlerunning{Approximation by Exponential type Neural Network Operators}        

\author{ Bajpeyi Shivam $^{*}$   \and
               A. Sathish Kumar   
}



\institute{S. Bajpeyi $^{*}$(corresponding author) \at
              Department of Mathematics, Visvesvaraya National Institute of Technology, Nagpur-440010
               \\
                          \email{shivambajpai1010@gmail.com } \vspace{4mm} \at  A. Sathish Kumar \at Department of Mathematics, Visvesvaraya National Institute of Technology, Nagpur-440010 \at \email{ mathsathish9@gmail.com}  }

\date{Received: date / Accepted: date}

\maketitle

\begin{abstract}

In the present article, we introduce and study the behaviour of the new family of exponential type neural network operators activated by the sigmoidal functions. We establish the point-wise and uniform approximation theorems for these NN (Neural Network) operators in $ \mathcal{C}[a,b].$
Further, the quantitative estimates of order of approximation for the proposed NN operators in $C^{(N)}[a,b]$ are established in terms of the modulus of continuity. We also analyze the behaviour of the family of exponential type quasi-interpolation operators in $ \mathcal{C}(\mathbb{R}^{+}).$ Finally, we discuss the multivariate extension of these NN operators and some examples of the sigmoidal functions.

\end{abstract}

\section{Introduction}
\label{intro}

The exponential sampling methods are powerful tool to solve problems arising in the areas of optical physics and engineering, precisely in the phenomena like Fraunhofer diffraction, light scattering etc. \cite{bertero,casasent,gori,ostrowsky}. It all started when a group of optical physicists and engineers Bartero, Pike \cite{bertero} and Gori \cite{gori} proposed a representation formula known as \textit{exponential sampling formula} (\cite{bardaro7}), for the class of Mellin band-limited function having exponentially spaced sample points. This exponential sampling formula is also considered as the Mellin-version of the well known \textit{Shannon sampling theorem} (\cite{shannon}). Butzer and Jansche (\cite{butzer5}) pioneered the mathematical study of the exponential sampling formula. They established a rigorous proof of this formula mathematically using the theory of Mellin transform and Mellin approximation, which was first studied separately by Mamedov \cite{mamedeo}, and then developed by Butzer and Jansche \cite{butzer3,butzer4,butzer5,butzer7}. We mention some of the work related to the theory of Mellin transform and approximation in \cite{bardaro1,bardaro2,bardaro3}. \par
Bardaro et al.\cite{bardaro7} introduced the generalized version of the exponential sampling series, acting on functions which are not necessarily Mellin band-limited, by replacing the $lin_{c}$ function in the exponential sampling formula, by more general kernel function. The generalized exponential sampling series is defined by
\begin{equation} \label{classical}
(S_{w}^{\chi}f)(x)= \sum_{k=- \infty}^{+\infty} \chi(e^{-k} x^{w}) f( e^{\frac{k}{w}}), \hspace*{0.3cm} \forall x \in \mathbb{R}^{+}, w>0
\end{equation}
where $ f:\mathbb{R}^{+} \rightarrow \mathbb{R} $ be any function for which the series is absolutely convergent. Subsequently, Bardaro et.al. extended the study of convergence of the above series in Mellin-Lebesgue spaces in \cite{bardaro11}. Recently, Balsamo and Mantellini \cite{Balsamo} considered the linear combination of these operators and analyzed its approximation results.\par

Let $C(\mathbb{R}^+$) denotes the space of all uniformly continuous functions defined on $\mathbb{R}^+,$ where $\mathbb{R}^{+}$ denotes the set of all positive real numbers. We call a function $f \in C(\mathbb{R}^+$) log-uniformly continuous on $\mathbb{R}^+$, if for any given
$\epsilon > 0,$ there exists $\delta > 0$ such that $|f(x) -f(y)| < \epsilon$ whenever $| \log x - \log y | \leq \delta,$ for any $x, y \in \mathbb{R}^{+}.$ We denote the space of all log-uniformly continuous functions defined on  $\mathbb{R}^{+}$ by $\mathcal{C}(\mathbb{R}^+).$ It is to mention that a log-uniformly continuous function need not uniformly continous. We consider $M(\mathbb{R}^{+})$ as the class of all Lebesgue measurable functions on $\mathbb{R}^+$ and $L^{\infty}(\mathbb{R}^{+}) $ as the space of all bounded functions on  $\mathbb{R}^{+}.$\par

Let $ L^p(\mathbb{R}^{+}), \ 1 \leq p < +\infty$ denotes the space of all the Lebesgue measurable functions defined on $\mathbb{R}^+,$ equipped with the usual norm $\Vert f \Vert_p$. For $c \in \mathbb{R}$, we define the space
$$X_c = \{f : \mathbb{R}^+ \rightarrow \mathbb{C} : f(\cdot)(\cdot)^{c-1} \in  L^1(\mathbb{R}^+)\}$$ equipped with the norm
$$\Vert f \Vert_{X_c} = \Vert f(\cdot)(\cdot)^{c-1} \Vert_1 = \int_0^{+\infty} |f(u)|u^{c-1}du.  $$
The Mellin transform of a function $f \in X_c$ is defined by
$$\hat{M}[f](s) := \int_0^{+\infty} u^{s-1}f(u)\ du \ ,(s = c + it, t \in \mathbb{R}).$$
We call a function $f \in X_{c} \cap C(\mathbb{R}^+),\  c \in  \mathbb{R},$ Mellin
band-limited in the interval $[-\eta, \eta],$ if $\hat{M}[f](c+iw) = 0$ for all $|w| > \eta ,\ \eta \in \mathbb{R}^+.$ \par
For a differentiable function $f : \mathbb{R}^+ \rightarrow \mathbb{C},$ the Mellin differential operator $\theta_c$, is defined as
$$\theta_cf(x) := xf'(x) + cf(x), \ \ \ \ \ x \in  \mathbb{R}^+  \mbox{and}\ c \in \mathbb{R}.$$
We set $  \theta_{0} f(x) := \theta f(x) .$ The Mellin differential operator of order $r \in \mathbb{N}$ is
defined by the following relation,
$$\theta_c^1 := \theta_c, \ \ \ \ \ \ \ \ \ \theta_c^{(r)} = \theta_c(\theta_c^{(r-1)}).$$

\section{Auxiliary results}
In this section, we give some preliminary definitions and related results which will be helpful to study the approximation properties of the family of operators $ (E^{\chi_{\sigma}}_{n}f).$ We begin with the definition of \textit{sigmoidal function}. A function $\sigma: \mathbb{R} \rightarrow \mathbb{R}$ is sigmoidal function if and only if $\displaystyle \lim_{x \rightarrow - \infty} \sigma(x)=0$ and $ \displaystyle  \lim_{x \rightarrow \infty} \sigma(x)=1.$ In the present context, we assume $\sigma(x)$ to be non-decreasing function with $\sigma(2) > \sigma(0)$ satisfying the following conditions,\\

$(1)$ $\sigma \in C^{(2)}(\mathbb{R})$ and concave on $\mathbb{R}^{+}.$ \par

$(2)$ $\sigma(x)= \mathcal{O}(|x|^{-1- \nu}),$ as $ x \rightarrow - \infty,$ for some $ \nu >0.$ \par

$(3)$ The function $(\sigma(x)- \frac{1}{2} \big)$ is an odd function.\\

Now, we define the density function activated by the sigmoidal function $\sigma$ by the following linear combination,
$$ \chi_{\sigma} (x) := \frac{1}{2} \big[ \sigma(\log x+1) - \sigma(\log x-1)\big] \ , \hspace{0.3cm} x \in \mathbb{R}^{+}.$$

The algebraic moment of order $\nu$ for the density function $ \chi_{\sigma}$ is defined as,
$$ m_{\nu}(\chi_{\sigma},u):= \sum_{k= - \infty}^{+\infty}  \chi_{\sigma}(e^{-k} u) (k- \log(u))^{\nu}, \hspace{0.5cm} \forall \ u \in \mathbb{R}^{+}.$$ \par
Similarly, the absolute moment of order $\nu$ can be defined as,
$$ M_{\nu}(\chi_{\sigma},u):= \sum_{k= - \infty}^{+\infty}  |\chi_{\sigma}(e^{-k} u)| |k- \log(u)|^{\nu},  \hspace{0.5cm} \forall \ u \in \mathbb{R}^{+}.$$ \par
We define \ $ \displaystyle M_{\nu}(\chi_{\sigma}):= \sup_{u \in \mathbb{R}^{+}} M_{\nu}(\chi_{\sigma},u). $\\

Now, we derive some basic results related to the density function $\chi_{\sigma}.$

\begin{lemma} \label{Lemma1}
The condition $\displaystyle \sum_{k=- \infty}^{+\infty} \chi_{\sigma}(e^{-k} u ) =1,$ holds for every $u \in \mathbb{R}^{+}$ and $k \in \mathbb{Z}.$
\end{lemma}

\begin{proof}
For any fixed $ x \in \mathbb{R}^{+}$ and $n \in \mathbb{N},$ we have
$$ \sum_{k= -n}^{n} \big( \sigma(\log x+1-k) -\sigma(\log x -k)\big) = \sigma(\log x+1+n) -\sigma(\log x -n). $$

Similarly,
$$  \sum_{k= -n}^{n} \big( \sigma(\log x-k) -\sigma(\log x -1-k)\big) = \sigma(\log x+n) -\sigma(\log x -n-1). $$

Using the definition of $\chi_{\sigma},$ we obtain
\begin{eqnarray*}
\displaystyle \sum_{k=- \infty}^{+\infty} \chi_{\sigma}(e^{-k} u ) &=& \frac{1}{2} \big[ \sigma(\log(u e^{-k})+1) -  \sigma(\log (u e^{-k})-1) \big ] \\
&=& \frac{1}{2} \big[ \sigma(\log u-k+1) -\sigma(\log u -k-1) \big] \\
&=& \frac{1}{2} \big[ (\sigma(\log u-k+1) - \sigma(\log u -k)) + (\sigma(\log u -k) -\sigma(\log u -k-1))\big] \\
&=& \frac{1}{2} \big[ \sigma(\log u+n+1) -\sigma(\log u -n) + \sigma(\log u+n) -\sigma(\log u -n-1) \big]
\end{eqnarray*}

As $ n \rightarrow + \infty,$ we obtain the desired result.

\end{proof}

\begin{lemma} \label{Lemma2}
Let $ x \in [a,b] ( \subset \mathbb{R}^{+})$ and $n \in \mathbb{N},\ k \in \mathbb{Z}. $ Then, we have
 $$ \displaystyle \sum_{k=\ceil{na}}^{\floor{nb}} \chi_{\sigma} (e^{-k} x^{n}) \geq \chi_{\sigma}(e).$$

\end{lemma}

\begin{proof}
Since $\displaystyle \sum_{k=- \infty}^{+\infty} \chi_{\sigma}(e^{-k} x^n ) =1, $ we have
$ \displaystyle \sum_{k=[na]}^{[nb]} \chi_{\sigma} (e^{-k} x^{n}) \leq  1.$ Now, we have
$$ \displaystyle \sum_{k=\ceil{na}}^{\floor{nb}} \chi_{\sigma} (e^{-k} x^{n}) = \sum_{k=[na]}^{[nb]} \chi_{\sigma} (|e^{-k} x^{n}|)
\geq \chi_{\sigma}(| e^{-\hat{k}} x^n |), \hspace{0.2cm} \forall x \in \mathbb{R}^{+}.$$

Now by choosing $ \hat{k} \in [\ceil{na}, \floor{nb}] \cap \mathbb{Z}$ such that $ 2 < |e^{-\hat{k}} x^{n}| < 3,$ we get $$ \chi_{\sigma} (|e^{-\hat{k}} x^{n}|) \geq \chi_{\sigma}(e) > 0, \ \mbox{as} \ \sigma(2) > \sigma(0).$$

This establishes the result.

\end{proof}

\textbf{Remark 1:} Using the condition $(1),$ it is easy to see that $\chi_{\sigma}(x)$ is non-increasing for $ x > 1 .$
\begin{lemma}\label{Lemma3}
For every $ \eta > 0$ and $u \in \mathbb{R}^+,$ the following condition holds
$$\displaystyle \lim_{n \rightarrow \infty} \sum_{|k - \log u | > n \eta} \chi_{\sigma}(e^{-k} u)=0. $$
\end{lemma}

\begin{proof} Let $\{ \log u \}$ denotes the fractional part of $\log u$ for every $ u \in \mathbb{R}^{+},$ i.e., $ \{\log u \}= \log u - [\log u],$ where, [ . ] denotes the\textit{ greatest integer function}. Define $ \hat{k} := k - [ \log u ].$ This gives $ |k - \log u|=| \hat{k}-\{\log u \}|.$ Now, for any fixed $\eta > 0,$ we have
\begin{eqnarray*}
\sum_{|k - \log u| > n \eta} \chi_{\sigma}(e^{-k} u) &=& \sum_{| \hat{k}-\{ \log u \} | > n \eta} \chi_{\sigma}(e^{-k} u) \\
& < &  \sum_{| \hat{k}| > n \eta +1} \chi_{\sigma}(e^{-k} u) \\
&=&\left( \sum_{  \hat{k} > n \eta  +1}\chi_{\sigma}(e^{-k} u)  + \sum_{  \hat{k}  < -(n \eta  +1)} \chi_{\sigma}(e^{-k} u) \right)\\
&:=& I_{1}+I_{2}.
\end{eqnarray*}
Using the condition $(2),$ we say that there exists $K > 0$ such that $$ \sigma(x) \leq K(|x|^{-1- \nu}), \hspace{0.3cm} \forall x < -L, \ \ \mbox{for some}\  L \in \mathbb{R}^+ \mbox{and} \  \nu >0 .$$
Since $\chi_{\sigma}(x) \leq \sigma(1+ \log x),$ we have
$$ I_{1} \leq \sum_{  \hat{k} > n \eta  +1} \sigma(1+\log u-k) \leq \sum_{  \hat{k} > n \eta  +1} \sigma(1+ \{\log u \}- \hat{k}) \ < K \sum_{  \hat{k} > n \eta  +1} |2- \hat{k}|^{-1- \nu}.$$
Moreover, from condition $(3)$ we obtain $$ \chi_{\sigma}(x) \leq 1 - \sigma(\log x -1)= \sigma(1- \log x), \hspace{0.1cm} \mbox{which gives}$$
$$ I_{2} \leq \sum_{ \hat{k} < -(n \eta  +1)} \sigma(1-\log u+k) \leq \sum_{ \hat{k} < -(n \eta  +1)} \sigma(1- \{\log u \}+ \hat{k}) \ \leq K \sum_{   \hat{k} < -(n \eta  +1)} |1+ \hat{k}|^{-1- \nu}.$$
As $\displaystyle \lim \ n\rightarrow + \infty,$ we get the desired result.

\end{proof}

\section{Neural Network Operators}
The feed-forward neural network operators (FNNs) with one hidden layer can be defined as
\begin{eqnarray*}
N_{n}(x)=\sum_{j=0}^{n}c_{j}\sigma(\langle a_{j}.x\rangle+b_{j}), \,\,\,\ x\in\mathbb{R}^{s},\,\,\,\ s\in\mathbb{N},
\end{eqnarray*}
where for $0\leq j\leq n,$ $b_{j}\in\mathbb{R}$ are thresholds, $a_{j}\in\mathbb{R}^{s}$ are the connection weights, $c_{j}\in\mathbb{R}$ are the coefficients, $\langle a_{j}.x\rangle$ is the inner product of $a_j$ and $x$, and $\sigma$ is the activation function of the network. The functions of one or several variables have been approximated by these NN (neural network) operators extensively from the last three decades. Cybenko \cite{Cybenko} and Funahashi \cite{Funahashi} have shown that any continuous function can be approximated on a compact set with uniform topology by the FNN-operators, using any continuous sigmoidal activation function. Further, Hornik et al. in \cite{Hornik}, proved that any measurable function can be approached with such a network. Cheney \cite{cheney} and Lenze \cite{lenze} used the idea of convolution kernel from the sigmoidal function to approximate the functions which is based on some results related to the theory of ridge functions. \par
The theory of NN operators has been introduced in order to study a constructive approximation process by neural networks. Anastassiou \cite{Anastassiou} pioneered the study of neural network approximation of continuous functions, where he extended the results proved by Cardalignet and Euvrard in \cite{euv}. Subsequently, Anastassiou analyzed the rate of approximation of the univariate and multivariate NN-operators activated by various sigmoidal functions in-terms of the modulus of smoothness (see\cite{Anastassiou,Anastassiou1,anas11a}). Costarelli and Spigler extended the results proved by Anastassiou in  (\cite{NN,2013d}). After this, many researchers analyzed the FNN-operators and established the various density results using different approaches, see (\cite{costaintro,costasurvey,Chen,Chui,Hahm,Leshno,Mhaskar}, etc.). Further, Costarelli and Spigler (\cite{kantNN}) studied the approximation of discontinuous functions by means of Kantorovich version of these NN-operators. Recently, the quantitative estimates of order of approximation was studied for kantorovich type NN operators in \cite{qntNN}.\par

Let $ f: [a,b] (\subset \mathbb{R}^+)  \rightarrow \mathbb{R}$ be any bounded function and $ n \in \mathbb{N}$ such that $ \ceil{na} \leq \floor{nb}.$ The exponential type neural network operators activated by the sigmoidal function $\sigma,$ are defined as
\begin{eqnarray*}
E^{\chi_{\sigma}}_{n}(f,x)=\frac{\displaystyle \sum_{k=\ceil{na}}^{\floor{nb}} f \big( e^{\frac{k}{n}} \big) \chi_{\sigma} (e^{-k} x^{n})}{\displaystyle \sum_{k=\ceil{na}}^{\floor{nb}} \chi_{\sigma} (e^{-k} x^{n})}, \hspace{0.5cm} x \in \mathbb{R^{+}}.
\end{eqnarray*}

Now, we derive the point-wise and uniform convergence results for the above family of operators $(E_{n}^{\chi_{\sigma}}f).$

\subsection{Approximation Results}

\begin{theorem} \label{Theorem1}
Let $ f: [a,b] (\subset \mathbb{R}^{+}) \rightarrow \mathbb{R}$ be any bounded function. Then, $(E_{n}^{\chi_{\sigma}}(f,x))$ converges to $ f(x)$ at every point of continuity of $f.$ Moreover, if $ f \in \mathcal{C}[a,b],$ then
$$ \lim_{n \rightarrow \infty} \| E_{n}^{\chi_{\sigma}}(f,.)-f \|_{\infty} = 0.$$
\end{theorem}

\begin{proof} We see that,
\begin{eqnarray*}
| E_{n}^{\chi_{\sigma}}(f,x) - f(x)| &=&
   \frac{\displaystyle \left| \sum_{k=\ceil{na}}^{\floor{nb}} \chi_{\sigma} (e^{-k} x^{n}) \big( f \big(e^{\frac{k}{n}} \big) - f(x) \big) \right| }{\displaystyle  \sum_{k=\ceil{na}}^{\floor{nb}} | \chi_{\sigma} (e^{-k} x^{n})|}  \\
&\leq & \frac{1}{|\chi_{\sigma}(e)|} \sum_{k=\ceil{na}}^{\floor{nb}} | \chi_{\sigma} (e^{-k} x^{n})| \Big| f \big(e^{\frac{k}{n}} \big) - f(x) \Big| \\
&\leq & \frac{1}{|\chi_{\sigma}(e)|} \sum_{|\frac{k}{n}- \log(x)| < \delta} |\chi_{\sigma} (e^{-k} x^{n})| \Big| f \big(e^{\frac{k}{n}} \big) - f(x) \Big| + \\ &&
\frac{1}{|\chi_{\sigma}(e)|} \sum_{|\frac{k}{n}- \log(x)| \geq \delta} | \chi_{\sigma} (e^{-k} x^{n})| \Big| f \big(e^{\frac{k}{n}} \big) - f(x) \Big|
:=I_{1}+I_{2}.
\end{eqnarray*}
First we evaluate $I_{1}$. Since $f \in \mathcal{C}[a,b],$ we have $ | I_{1}| <  \frac{\epsilon }{|\chi_{\sigma}(e)|} .$
Now for $I_{2},$ using Lemma \ref{Lemma3}, we have
   $$ |I_{2}| \leq \frac{2 \|f \|_{\infty}}{|\chi_{\sigma}(e)|} \epsilon. $$

On combining the estimates $ I_{1}-I_{2},$ we establish the desired result.\\

Now, we obtain the order of convergence of  $(E_{n}^{\chi_{\sigma}}f)$ for the functions belonging to the class of \textit{log-holderian functions} of order $\lambda,$ for $ 0< \lambda \leq 1$ (see \cite{bardaro7}), which is defined as follows
$$ L_{\lambda}(f)= \{ f: \mathbb{R}^+ \rightarrow \mathbb{R} \ : | f(x)-f(y)| <  H | \log x - \log y |^{\lambda};  \hspace{0.11cm}x,y,H \in \mathbb{R}^+ \}. $$

\begin{theorem} \label{Theoem2}
Let $ f \in L_{\lambda}(f)$ and $ M_{\lambda}(\chi_{\sigma}) < + \infty ,$ for some $ \lambda >0.$ Then, we have
$$ \lim_{n\rightarrow \infty} \| E_{n}^{\chi_{\sigma}}(f,.)-f \|_{\infty} = \mathcal{O} (n^{- \lambda}).$$
\end{theorem}

\begin{proof} Since $ f \in L_{\lambda}(f),$ then \ $\big| f \big(e^{\frac{k}{n}} \big) - f(x) \big| \leq H  \big|\frac{k}{n}- \log(x) \big|^{\lambda},$ whenever $ \big| \frac{k}{n}- \log(x) \big| < \delta,\ \delta >0.$ Now, we have
\begin{eqnarray*}
| E_{n}^{\chi_{\sigma}}(f,x) - f(x)| &\leq & \frac{1}{|\chi_{\sigma}(e)|} \sum_{k=\ceil{na}}^{\floor{nb}} | \chi_{\sigma} (e^{-k} x^{n})| \Big| f \big(e^{\frac{k}{n}} \big) - f(x) \Big| \\
&\leq & \frac{1}{|\chi_{\sigma}(e)|} \sum_{|\frac{k}{n}- \log(x)| < \delta} |\chi_{\sigma} (e^{-k} x^{n})| \Big| f \big(e^{\frac{k}{n}} \big) - f(x) \Big| + \\&&
\frac{1}{|\chi_{\sigma}(e)|} \sum_{|\frac{k}{n}- \log(x)| \geq \delta} |\chi_{\sigma} (e^{-k} x^{n})| \Big| f \big(e^{\frac{k}{n}} \big) - f(x) \Big|\\
&:=& I_{1}+I_{2}.
\end{eqnarray*}

It easily follows that, $|I_{1}| < \frac{H n^{- \lambda}}{|\chi_{\sigma}(1)|} M_{\lambda}(\chi_{\sigma}). $ Next, we estimate $I_{2}.$
\begin{eqnarray*}
|I_{2}| & \leq & \frac{2 \|f\|_{\infty}}{|\chi_{\sigma}(e)|} \sum_{|\frac{k}{n}- \log(x)| \geq \delta} |\chi_{\sigma} (e^{-k} x^{n})|\\
&\leq & \frac{1}{|\chi_{\sigma}(e)|}\frac{2 \|f\|_{\infty}}{n^{\lambda} \delta^{\lambda}} M_{\lambda}(\chi_{\sigma}).
\end{eqnarray*}
As $ M_{\lambda}(\chi_{\sigma}) < + \infty,$ it completes the proof.
\end{proof}

\subsection{Quantitative Estimates}
In this section, we obtain a quantitative estimate of order of convergence for the family of operators $(E_{n}^{\chi_{\sigma}}f)$ for a particular sigmoidal function. There are several examples of the sigmoidal function $\sigma$ satisfying all the assumptions of the presented theory (see \cite{anas11a,anas11d,anas11b,anas11c,NN,costab,cheang,cao,Cybenko}). We begin with an example of such sigmoidal function known as  \textit{hyperbolic tangent} sigmoidal function \cite{anas11d}, is defined as
$$ \sigma_{h}(x):= \frac{1}{2} (tanh \ x+1), \hspace{0.2cm} x \in \mathbb{R}.$$

The corresponding density function activated by the above sigmoidal function is as follows,
$$ \chi_{\sigma}(x) = \frac{1}{2}\Bigg[ \frac{x^{2}(e^4 -1)}{x^{2}(1+e^4 + e^{2} x^{2})+ e^2 } \Bigg] \ , \hspace{0.3cm} \forall x \in \mathbb{R}^{+}.$$

Now, let $ 0 < \nu < 1$ and $ n \in \mathbb{N}.$ Indeed,
\begin{eqnarray*}
\sum_{\big|\frac{k}{n}- \log(x)\big|  \geq n^{1- \nu}} \chi_{\sigma} (e^{-k} x^{n}) & = &  \sum_{\big|\frac{k}{n}- \log(x)\big|  \geq n^{1- \nu}} \chi_{\sigma} (|e^{-k} x^{n}|) \\
& < & \frac{(e^{4} - 1)}{2 e^{2}} \int_{n^{1- \nu}}^{+ \infty} \frac{1}{x^{2}} dx \\
&=& \frac{(e^{4} - 1)}{2 e^{2}} \frac{1}{n^{1- \nu}}\\
&=& \frac{3.6268}{n^{1- \nu}}.
\end{eqnarray*}

Using the above estimate, we establish a quantitative estimate of order of convergence in $ C([a,b])$ using the notion of \textit{logarithmic modulus of continuity}. \\

The logarithmic modulus of continuity is defined as
$$ \omega(f,\delta):= \sup \{|f(x)-f(y)| ;\  \mbox{whenever} \  |\log (x)-\log (y)| \leq \delta,\  \ \delta \in \mathbb{R}^{+}\} .$$
The properties of logarithmic modulus of continuity can be seen in \cite{bardaro9}.

\begin{theorem} \label{Theorem3}
Let $[a,b] \subset \mathbb{R}^+$ and $f \in C([a,b]).$ Then, for $0 < \nu <1,$ we have the following estimate,
$$ | E_{n}^{\chi_{\sigma}}(f,x) - f(x)| \leq (4.14925) \left( \omega \left(f, \frac{1}{n^{\nu}} \right) + 7.2536 \|f \|_{\infty} n^{(\nu -1)} \right),$$ where, $ 0 < \nu <1, \ \ n \in \mathbb{N}.$
\end{theorem}

\begin{proof} We have
\begin{eqnarray}
\displaystyle \nonumber
 |E_{n}^{\chi_{\sigma}}(f,x) - f(x)| &=& \frac{\displaystyle \left| \sum_{k=\ceil{na}}^{\floor{nb}}  \chi_{\sigma} (e^{-k} x^{n}) \ \big( f \big(e^{\frac{k}{n}} \big) - f(x) \big) \right| }{\displaystyle \sum_{k=\ceil{na}}^{\floor{nb}} |\chi_{\sigma} (e^{-k} x^{n})|}  \\
&\leq & \frac{1}{|\chi_{\sigma}(e)|} \sum_{k=\ceil{na}}^{\floor{nb}} | \chi_{\sigma} (e^{-k} x^{n})| \Big| f \big(e^{\frac{k}{n}} \big) - f(x) \Big|
\end{eqnarray}
Now, we can write
\begin{eqnarray*}
\Bigg|\sum_{k=\ceil{na}}^{\floor{nb}} \Big(f \big( e^{\frac{k}{n}} \big)-f(x) \Big) \ \chi_{\sigma} (e^{-k} x^{n}) \Bigg| & \leq & \sum_{k=\ceil{na}}^{\floor{nb}} | f \big( e^{\frac{k}{n}} \big)-f(x) | \ |\chi_{\sigma} (e^{-k} x^{n})| \\
&=& \sum_{|\frac{k}{n}- \log(x)| < \frac{1}{n^{\nu}}} |\chi_{\sigma} (e^{-k} x^{n})| \big| f \big(e^{\frac{k}{n}} \big) - f(x) \big| + \\&&
\sum_{|\frac{k}{n}- \log(x)| \geq \frac{1}{n^{\nu}}} | \chi_{\sigma} (e^{-k} x^{n}) | \big| f \big( e^{\frac{k}{n}} \big)-f(x) \big| \\
 & \leq & \sum_{|\frac{k}{n}- \log(x)| < \frac{1}{n^{\nu}}} |\chi_{\sigma} (e^{-k} x^{n})| \ \omega(f,\frac{1}{n^{\nu}}) + 2 \|f\|_{\infty} \sum_{|\frac{k}{n}- \log(x)| \geq \frac{1}{n^{\nu}}} |\chi_{\sigma} (e^{-k} x^{n})|  \\
& \leq & \omega \left(f,\frac{1}{n^{\nu}} \right) + 2 \|f\|_{\infty}
\Big( \frac{3.6268}{n^{1- \nu}} \Big).
\end{eqnarray*}

Using this estimate in $(2),$ we establish the result.

\end{proof}

In order to derive the higher order of approximation, we use the Taylor's formula in terms of Mellin's derivatives as the Mellin's analysis is the suitable frame to put the rigorous theory of the exponential sampling (see \cite{mamedeo,bardaro7} ).

\begin{theorem} \label{Theorem4}
Let $[a,b] \subset \mathbb{R}^+$ and $f \in C^{(2)}([a,b]).$ Then, for $0 < \nu <1,$ the following estimate holds
$$ |E_{n}^{\chi_{\sigma}}(f,x) - f(x)| \leq \ (4.14925)\ \left \{ \sum_{i=1}^{2}  \frac{\|\theta^{(i)}f \|_{\infty}}{i !} \left( n^{- i \nu } + |b-a|^{i} \frac{3.6268}{n^{1- \nu}} \right)+\frac{\omega(\theta^{(2)}f,n^{-\nu})}{2 n^{2 \nu}} \ + \frac{(3.6268)}{n^{1- \nu}} \|\theta^{(2)}f \|_{\infty} (b-a)^{2} \right \}.$$
\end{theorem}

\begin{proof} Let $f \in C^{(2)}[a,b].$ Then, by the Taylor's formula in terms of Mellin's derivatives with integral form of remainder, we have
$$  f(e^{u})= f(x)+ (\theta f)(x) \big( u - \log x \big)+  \int_{\log x}^{u} (\theta^{(2)}f)(e^\xi) \big(u - \xi \big) d \xi, \hspace{0.2cm}\mbox{where,} \ \xi \in (\log x, u).$$
Now, we can write
$$ \left( \sum_{k=\ceil{na}}^{\floor{nb}} f \big( e^{\frac{k}{n}} \big) \chi_{\sigma} (e^{-k} x^{n}) - f(x) \sum_{k=\ceil{na}}^{\floor{nb}} \chi_{\sigma} (e^{-k} x^{n}) \right)$$
\begin{eqnarray*}
 &=& (\theta f)(x) \sum_{k=\ceil{na}}^{\floor{nb}} \chi_{\sigma} (e^{-k} x^{n}) \Big( \frac{k}{n} - \log x \Big)+ \sum_{k=\ceil{na}}^{\floor{nb}} \chi_{\sigma} (e^{-k} x^{n}) \int_{\log x}^{\frac{k}{n}} (\theta^{(2)}f)(e^\xi) \Big(\frac{k}{n}- \xi \Big) d \xi \\
 &=& \sum_{i=1}^{2} \frac{(\theta^{(i)}f)(x)}{i!} \left(  \sum_{[na]}^{[nb]} \chi_{\sigma} (e^{-k} x^{n}) \right) \Big( \frac{k}{n} - \log x \Big)^{i} +\sum_{[na]}^{[nb]} \chi_{\sigma} (e^{-k} x^{n}) \int_{\log x}^{\frac{k}{n}} \big( (\theta^{(2)}f)(e^\xi) - (\theta^{(2)}f)(x) \big) \Big(\frac{k}{n}- \xi \Big) d \xi.\\
&=& \sum_{i=1}^{2} \frac{(\theta^{(i)}f)(x)}{i!} \left(  \sum_{[na]}^{[nb]} \chi_{\sigma} (e^{-k} x^{n}) \right) \Big( \frac{k}{n} - \log x \Big)^{i} +I \ \mbox{(say)}.
\end{eqnarray*}
Now, $I$ can be written as
$$ I =  \left( \sum_{ |\frac{k}{n}- \log(x)| < \frac{1}{n^{\nu}}}+\sum_{ |\frac{k}{n}- \log(x)| \geq \frac{1}{n^{\nu}}} \right)\chi_{\sigma} (e^{-k} x^{n}) \int_{\log x}^{\frac{k}{n}} \big( (\theta^{(2)}f)(e^\xi) - (\theta^{(2)}f)(x) \big) \Big(\frac{k}{n}- \xi \Big) d \xi . $$
In the view of modulus of continuity and the estimate of $\displaystyle \sum_{ |\frac{k}{n}- \log(x)| \geq \frac{1}{n^{\nu}}}  |\chi_{\sigma} (e^{-k} x^{n})|,$ we obtain

$$ |I| \leq \left( \frac{1}{2 n^{2 \nu}} \ \omega(\theta^{(2)}f,n^{-\nu})+ \frac{(3.6268)}{n^{1- \nu}} \|\theta^{(2)}f \|_{\infty} (b-a)^{2} \right).$$
Further, we estimate $\displaystyle \sum_{k=\ceil{na}}^{\floor{nb}} | \chi_{\sigma} (e^{-k} x^{n}) |  \big| \frac{k}{n} - \log x \big |^{i}$
\begin{eqnarray*}
 &=& \sum_{| \frac{k}{n} - \log x | \leq n^{- \nu}} | \chi_{\sigma} (e^{-k} x^{n}) | \  \big| \frac{k}{n} - \log x \big |^{i} + \sum_{| \frac{k}{n} - \log x | \geq n^{- \nu}} | \chi_{\sigma} (e^{-k} x^{n}) | \  \big| \frac{k}{n} - \log x \big | ^{i} \\
&\leq & n^{- i \nu } \sum_{| \frac{k}{n} - \log x | \leq n^{- \nu}} | \chi_{\sigma} (e^{-k} x^{n}) | + |b-a|^{i} \sum_{| \frac{k}{n} - \log x | \geq n^{- \nu}} | \chi_{\sigma} (e^{-k} x^{n}) | \\
& \leq & n^{- i \nu } + |b-a|^{i} \frac{3.6268}{n^{1- \nu}}.
\end{eqnarray*}

On combining all the above estimates, we get the desired result.

 \end{proof}

\end{proof}

\textbf{Remark 2:} Let $f \in C^{(N)}([a,b]),\ N \in \mathbb{N}.$ Then, the higher order Taylor's formula in terms of Mellin's derivatives with integral form of remainder is given by
$$ \displaystyle f(e^{u})= f(x)+ (\theta f)(x) \big( u - \log x \big)+ \frac{(\theta^{(2)} f)(x)}{2!} (u - \log(x))^{2}+...
+  \int_{\log x}^{u} \frac{(\theta^{(n)}f)(e^\xi)}{(n-1)!} \big(u - \xi \big)^{n-1} d \xi$$
where, $\xi \in (\log x, u).$
The proof of this Mellin's Taylor formula can be easily obtained by using the fundamental theorem of integral calculus.

\section{Exponential Type Quasi-Interpolation Operators}

In the present section, we define the notion of exponential analogue of quasi-interpolation operators (\cite{anas11d,anas12,NN,cao}). Let $\sigma$ be the sigmoidal function satisfying all the conditions $(1)-(3)$ and $\chi_{\sigma}(.)$ be the corresponding density function generated by the sigmoidal function $\sigma$. Then, for every $x \in \mathbb{R}^{+}, k \in \mathbb{Z}$ and $f: \mathbb{R}^{+} \rightarrow \mathbb{R},$ the exponential type quasi-interpolation operators are defined as,
$$ Q_{n}(f,x) :=  \displaystyle \sum_{k = - \infty}^{\infty} f \big( e^{\frac{k}{n}} \big) \chi_{\sigma} (e^{-k} x^{n}) , \hspace{0.3cm} \forall \ n \in \mathbb{N}.$$
The above operators are well-defined for the class of all real valued bounded functions defined on $\mathbb{R}^+.$ \\

First, we derive the point-wise and uniform convergence theorem for these operators in $\mathcal{C}(\mathbb{R}^{+}).$
\begin{theorem} \label{Theorem5}
Let $ f: \mathbb{R}^{+} \rightarrow  \mathbb{R}$ be any bounded function. Then, $(Q_{n}^{\chi_{\sigma}}(f,x))$ converges to $ f(x)$ at every point of continuity of $f.$ Moreover, if $ f \in \mathcal{C}(\mathbb{R}^+),$ then
$$ \lim_{n \rightarrow \infty} \| Q_{n}^{\chi_{\sigma}}(f,.)-f \|_{\infty} = 0.$$

\end{theorem}

\begin{proof}
The proof follows in the similar manner as discussed in Theorem \ref{Theorem1}, with the application of Lemma \ref{Lemma1}.

\end{proof}

Moreover, we obtain a similar type of quantitative estimate of order of approximation for the family of operators $(Q_{n}^{\chi_{\sigma}}f),$ provided $f$ is bounded on $\mathbb{R}^+.$

\begin{theorem}\label{Theorem6}
Let $ f \in C(\mathbb{R}^+)$ be any bounded function. Then, the following estimate holds
$$ | Q_{n}^{\chi_{\sigma}}(f,x) - f(x)| \leq  \left ( \omega \left( f, \frac{1}{n^{\nu}}\right) + 7.2536 \|f \|_{\infty} n^{(\nu -1)} \right )$$ where, $ 0 < \nu <1, \ \ n \in \mathbb{N}.$
\end{theorem}

\begin{proof} The proof can be obtained by following the proof of Theorem \ref{Theorem3}.
\end{proof}

\section{Multivariate Exponential Type Neural Network Operators}
In this section, we extend the concept of exponential type neural network operators in the multivariate setting. Let $ \textbf{x} \in (\mathbb{R}^{+})^{N},\ N \in \mathbb{N}.$
Then, the multivariate density function activated by the sigmoidal function $\sigma$ satisfying the assumptions $(1)-(3),$ is defined as
$$ \bar{\chi}^{N}_{\sigma}(\textbf{x}) := \prod_{i=1}^{N} \chi_{\sigma}(x_{i}),\hspace{0.2cm} x_{i} \in \mathbb{R}^{+}$$
where, each $\chi_{\sigma}(x_{i})$ is the univariate density function defined on $\mathbb{R}^{+}.$ As the multivariate density function $ \bar{\chi}^{N}_{\sigma}(\textbf{x})$ is the finite product of univariate density functions $\chi_{\sigma}(x_{i}),$ it can easily be shown that $ \bar{\chi}^{N}_{\sigma}(\textbf{x})$ also satisfies the  multivariate extensions of Lemma $1,2$ and $3$ (see \cite{anas11b,anas11c}). Now, the multivariate exponential type neural network operators can be defined as follows,
$$ E^{\bar{\chi}^{N}_{\sigma}}_{n}(f,\textbf{x})=\frac{\displaystyle \sum_{k_{1}=\ceil{na}}^{\floor{nb}}...\sum_{k_{N}=\ceil{na}}^{\floor{nb}} f \big( e^{\frac{\hat{k}}{n}} \big) \ \hat{\chi}^{N}_{\sigma} (e^{- \hat{k}} \textbf{x}^{n})}{\displaystyle \sum_{k_{1}=\ceil{na}}^{\floor{nb}}...\sum_{k_{N}=\ceil{na}}^{\floor{nb}} \hat{\chi}^{N}_{\sigma} (e^{- \hat{k}} \textbf{x}^{n})} \ , \hspace{0.3cm} n \in \mathbb{N}$$
where, $\hat{k}= (k_{1},k_{2},...,k_{N}) \in \mathbb{Z}^N.$
One can study the point-wise and uniform approximation theorems for the multivariate extension of exponential type neural network operators activated by sigmoidal functions, by following the same proof technique as discussed earlier.
\section{Examples}
In this section, we present few examples of the sigmoidal function $\sigma(x)$ which satisfy the assumptions of the proposed theory.  We begin with an example of sigmoidal function which is known as \textit{logistic function} (\cite{anas12}) and defined as,
$$ \sigma_{l}(x)= \frac{1}{1+e^{-x}}, \hspace{0.2cm} x \in \mathbb{R}.$$
It is easy that the logistic function satisfies all the conditions $(1)-(3)$ of section $1$ (see \cite{NN}). \par

Another example of such smooth sigmoidal function is \textit{hyperbolic tangent} sigmoidal function (\cite{anas11d}), which is defined as,
$$ \sigma_{h}(x):= \frac{1}{2} (tanh \ x+1), \hspace{0.2cm} x \in \mathbb{R}.$$

Moreover, we can construct the example of sigmoidal function satisfying all the assumptions $(1)-(3)$ with the help of well-known \textit{B-spline functions} (\cite{but,NN}) of order $n,$ by defining as follows,
$$ \sigma_{M_{n}}(x):= \int_{- \infty}^{x} M_{n}(t) dt, \hspace{0.3cm} x \in \mathbb{R}.$$
It is clear that $\sigma_{M_{n}}(x)$ is non-decreasing on $\mathbb{R}^+$ and $0 \leq \sigma_{M_{n}}(x) \leq 1 $ for every $n \in \mathbb{N}$ and $x \in \mathbb{R}^+ .$
Indeed,

\begin{equation*}
\sigma_{M_{1}}(x) =
     \begin{cases}
      {0,} &\quad\text{if,} \ \ \ \ {x < - \frac{1}{2}} \\
     {x+ \frac{1}{2}} &\quad\text{if,} \ {- \frac{1}{2} \leq x \leq \frac{1}{2}}\\
     {1,} &\quad\text{if,} \ \ \ \ {x > \frac{1}{2}}\\
   \end{cases}
\end{equation*}

This function is also known as \textit{ramp-function} (see \cite{anas11d,anas11b}) which is an example of a discontinuous sigmoidal function. \\

Similarly, we have
\begin{equation*}
\sigma_{M_{2}}(x) =
     \begin{cases}
      {\frac{(1+x)^{2}}{2}} &\quad\text{if,} \ \ \ \ {-1< x < 0} \\
     {\left(1- \frac{(x-1)^{2}}{2} \right)} &\quad\text{if,} \ \ \ \ \ \ {0 \leq x < 1.}\\
   \end{cases}
\end{equation*} \\

\textbf{Conclusion:}  In this paper, we introduced and studied the behaviour of the new family of Exponential type neural network operators activated by the sigmoidal functions. We established the point-wise and uniform approximation theorems for these NN (Neural Network) operators in $ \mathcal{C}[a,b].$ We also studied the order of convergence for these operators in the class of \textit{log-holderian} functions.
Further, the quantitative estimates of order of approximation for the proposed NN operators in $C^{N}[a,b]$  established in terms of logarithmic modulus of continuity. We also analyzed the behaviour of the family of exponential type quasi-interpolation operators in $ \mathcal{C}(\mathbb{R}^+)$ and the quantitative estimate of order of convergence in $C(\mathbb{R}^+).$ Finally, we propose the multivariate extension of exponential type NN operators and few examples of the sigmoidal functions to which the present theory can be applied.

{}

\end{document}